\documentclass[12pt]{amsart}

\usepackage{amssymb,amsmath,amscd}
\usepackage{graphicx}
\usepackage{mathrsfs}
\usepackage{amsthm}
\usepackage{latexsym}
\usepackage{enumitem}

\usepackage{stmaryrd}

\usepackage{palatino}
\usepackage[T1]{fontenc}  

\usepackage[margin=1.25in]{geometry}

\theoremstyle{plain}

\usepackage{enumitem}
\setlist{noitemsep,topsep=6pt,parsep=0pt,partopsep=0pt}

\usepackage{algorithm}
\usepackage{algpseudocode}

\makeatletter \def\BState{\State\hskip-\ALG@thistlm} \makeatother

\usepackage{etoolbox} \makeatletter
\patchcmd{\@maketitle}{\begin{center}}{\begin{flushleft}}{}{}
\patchcmd{\@maketitle}{\begin{tabular}[t]{c}}{\begin{tabular}[t]{@{}l}}{}{}
\patchcmd{\@maketitle}{\end{center}}{\end{flushleft}}{}{}

\DeclareMathAlphabet{\mathpzc}{OT1}{pzc}{m}{it}

\newtheorem{thm}{Theorem}[section]
\newtheorem{lem}[thm]{Lemma}

\newtheorem{cor}[thm]{Corollary}
\newtheorem{dfn}[thm]{Definition}

\newtheorem{obs}[thm]{Observation}

\def\mtree{\Gamma}

\def\md{\mathbf{\alpha}}

\def\leaves{\Lambda}
\def\leaf{\lambda}

\newcommand{\style}[1]{{\emph{#1}}}  

\newcommand{\real}{{\mathbb R}}

\newcommand{\dist}{{f}}             
\newcommand{\Dist}{{\mathfrak{D}}}  

\newcommand{\inv}{^{-1}}

\newcommand{\ucat}{{{\sf ucat}}}         

\newcommand{\uni}{{u}}              

\title{Minimal Unimodal Decomposition on Trees}
\author{Yuliy Baryshnikov}
\address{Departments of Mathematics and Electrical \& Computer Engineering, University of Illinois}
\email{ymb@illinois.edu}
\thanks{Work supported by the National Science Foundation via DMS-1622370.}
\author{Robert Ghrist}
\address{Departments of Mathematics and Electrical \& Systems Engineering, University of Pennsylvania}
\email{ghrist@math.upenn.edu}
\thanks{Work supported by the Office of the Assistant Secretary of Defense Research \& Engineering through ONR N00014-16-1-2010.}

\begin{document}
\maketitle
\begin{abstract}
The decomposition of a density function on a domain into a minimal sum of unimodal components is a fundamental problem in statistics, leading to the topological invariant of \style{unimodal category} of a density. This paper gives an efficient algorithm for the construction of a minimal unimodal decomposition of a tame density function on a finite metric tree.
\end{abstract}

\section{Introduction}

\subsection{Motivation}

In data analysis, the operation of clustering is fundamental. At its base is a problem wedged between geometry and topology: given a set of points and a notion of distance or proximity among them, compute a parsimonious division into sets of points that are mutually {\em close}.  In practice, this is a subtle question that depends sensitively on the models of proximity for the input and the desired nature of the output \cite{Kleinberg,CM}.

The discretized nature of this problem obscures the more topological (and ancient) problem of minimal decomposition of a space into simple pieces or, suggestively, {\em modal domains}. For a geometric domain $D\subset\mathbb{E}^n$, one might ask for the minimal number of convex pieces into which $D$ can be decomposed, minimal number providing an unambiguous descriptor of geometric complexity. For a topological space $X$, the minimal mode is not a convex domain but rather a \style{contractible} subset $U\subset X$, meaning that $U$ has the homotopy type of a point. The minimal number of such subsets covering $X$ is called the \style{geometric category} of $X$ and is a homeomorphism invariant of $X$. There is a family of related topological categories, the most important being the \style{Lusternik-Schnirelmann category} (using subsets nullhomotopic in $X$) which branches quickly into homotopy theory \cite{Cornea+}. Other related notions of minimal topological decomposition include the classical \style{sectional category} (or \style{Schwarz genus}) of a fibration and the more modern variant of \style{topological complexity} of path planning \cite{Farber}.

This paper concerns a weighted version of these problems adapted to statistics. Given a {\em density} $f\colon X\to[0,\infty)$, what is the minimal number of {\em modes} into which it can be decomposed as a sum? In the geometric version of this problem, one natural notion of a mode is a Gaussian, and the problem of minimal approximation of a density as a sum of Gaussians is well-studied on one-dimensional domains \cite{B,E,RF}. This choice of a Gaussian as fundamental mode is somewhat geometrically rigid, analogous to the clustering of a space into convex pieces. One can imagine other basis unimodal distributions \cite{Ka,Ke}.

This paper concerns the coordinate-free topological version of the problem of minimal decompositions of densities over a space --- a statistical analogue of the category of a domain.  The resulting invariant is called the \style{unimodal category} \cite{BG,EAT}.

\subsection{Unimodal category}

For $X$ a topological space, let $\Dist=\Dist(X)$ denote the set of all compactly supported continuous functions $\dist:X\to[0,\infty)$. Such a function $\uni\in\Dist$ is \style{unimodal} if the upper excursion sets $\uni^c=\uni\inv([c,\infty))$ have the homotopy type of a point for all $0<c\leq M$ and are empty for all $c>M$. Such a $\uni$ has $M$ as its maximal value.

We will refer to the nonempty upper excursion sets $\uni^c\subset X$ as being \style{contractible}, though it must be clarified that such sets are contractible {\em in themselves} as opposed to being merely contractible in $X$. The latter would be more in line with the definitions used in Lusternik-Schnirelmann theory, but is less relevant for most applications (where $X=\real^n$).

\begin{dfn}
\label{def:unicat}
The \style{unimodal category} of $\dist\in\Dist(X)$ is the minimal number $\ucat(\dist)$ of unimodal distributions $\uni_\alpha, \alpha=1,\ldots,\ucat(\dist)$ on $X$ such that $\dist$ is the pointwise sum of the collection $\{\uni_\alpha\}$.
\end{dfn}

In the data analysis interpretation of the unimodal functions, where the mode corresponds to {\em signal}, and the spread of the density around it, to {\em noise}, it makes sense to assume some similarities of the noise generating mechanisms for different modes. In the world of Gaussian distribution this leads to assumption of the fixed, or slowly varying covariant form. A parsimonious, homeomorphism-invariant version would assume a much weaker formulation, which nonetheless strengthens significantly the notion of unimodal category, as follows:

\begin{dfn}
\label{def:str_ucat}
The \style{strong unimodal category} of $\dist\in\Dist(X)$ is the minimal number $\ucat(\dist)$ of distributions $\uni_\alpha, \alpha=1,\ldots,\ucat(\dist)$ on $X$ that sum up to $\dist$ such that any intersection of the upper excursion sets is either empty or contractible.
\end{dfn}

Since unimodal functions remain unimodal under a homeomorphic change of coordinates, the unimodal category is a coordinate-free invariant of a density. This makes it of significant potential use in applications where data is collected from noisy or otherwise uncertainly located samples. The initial paper on the subject gave a constructive algorithm for computing the unimodal category on $\mathbb{R}^1$, along with generalizations to unimodal categories based on pointwise norms rather than addition \cite{BG}. The thesis of Govc showed just how subtle and difficult the problem of computation of these invariants is in higher dimensions \cite{Govc}. Recent applications of the unimodal category in 1-d by Huntsman are currently being applied to problems of mixture estimation in statistics \cite{Huntsman}.

We believe that {\em strong} unimodal category would prove more amenable to analysis, but postpone this line of research till later, noting only that for our model, where the underlying topological space is a metric tree, the notion of strong unimodal category is identical with the standard one.

\subsection{Contribution}
The primary contribution of this note is the presentation and proof of correctness of an efficient greedy algorithm for the computation of a minimal unimodal decomposition of a density over a tree, this yielding the unimodal category. More than simply computing a topological invariant, this method permits identification of ``essential'' modes which are local maxima for {\em any} minimal unimodal decomposition. We believe that this extension will permit novel applications of the unimodal category and point the way to computational methods on suitably restricted higher dimensional domains.

As a perhaps fanciful toy problem for the sake of motivation, consider the following scenario. Suppose one wishes to detect radiological substances by means of crude sensors mounted on vehicles in traffic on city streets. At any given time, the network of vehicles returns a sampling of a distribution on a graph (radioactivity levels restricted to the idealized 1-d cell complex of city streets). It is perhaps known that there are certain expected modes --- at, say, hospitals, research facilities, and universities. In a dense urban environment, one might have difficulty distinguishing true modes from interference modes, and the problem of noise in position measurement is a further complicating feature. It is in such circumstances that a minimal unimodal decomposition may assist as an unambiguous warning of a significant unexpected mode.

\section{Setup}
\label{sec:setup}
\subsection{Linearization}
\label{sec:assumptions}
Let $\mtree$ be a finite metric tree, that is, a 1-dimensional compact contractible metric space stratified into 0-dimensional {\em vertices} and 1-dimensional (open) {\em edges}. All subtrees of $\mtree$ are assumed to inherit both the cell structure and the metric of $\mtree$. The subtrees of $\mtree$ form a finite lattice.

In order to work with finite unimodal decompositions, we henceforth assume that $\Dist(\mtree)$ is restricted to functions with a finite number of critical points. Given such an $f\in\Dist(\mtree)$, we can assume that the restriction of the function to each edge is weakly monotonic, by adding new vertices at local maxima and minima. Using the restriction of $f$ to each edge (or an affine linear function of it) as a coordinate, we can also assume that $f$ is \style{edge-linear}, that is, $f$ restricted to each edge is an affine function of that edge with respect to the internal metric structure: see Figure \ref{fig:unimodalfunction}.

\begin{figure}[htbp]
\begin{center}
\includegraphics[width=5in]{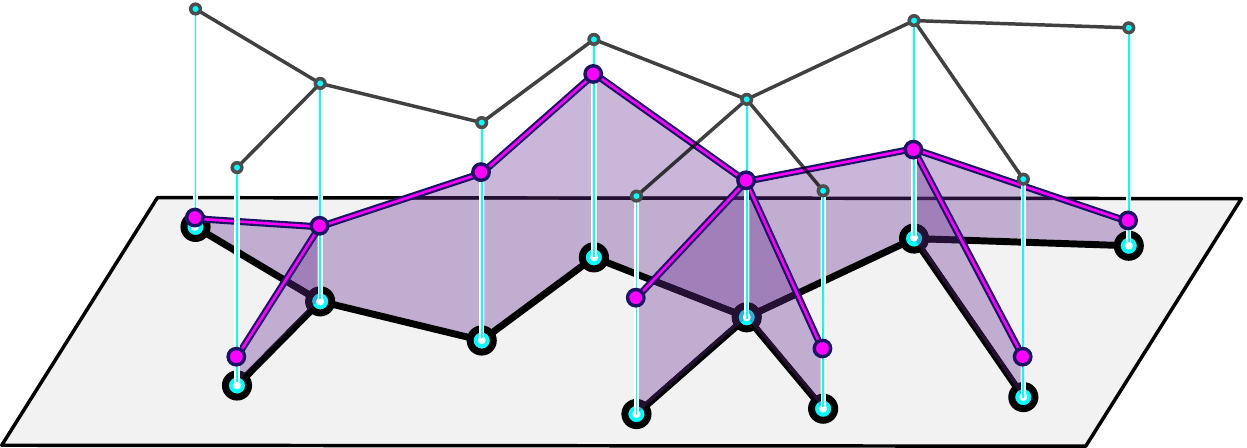}
\caption{Example of an edge-linear unimodal function on a tree.}
\label{fig:unimodalfunction}
\end{center}
\end{figure}

\begin{lem}
\label{lem:contraction}
If $f$ is constant on an edge $e$ of $\mtree$, then the operation of contraction by $e$ taking $\mtree$ to the tree $\mtree/e$ sends unimodal decompositions of $f$ on $\mtree$ to unimodal decompositions of $f$ on $\mtree/e$ preserving number of modes. This operation is reversible, taking any unimodal decomposition of $f$ on $\mtree/e$ to a mode-preserving decomposition of $f$ on $\mtree$.
\end{lem}

\begin{lem}
\label{lem:linearmodes}
If $f=\sum f_\md$ is a unimodal decomposition of an edge-linear $f\in\Dist(\mtree)$, then its components $f_\md$ can be chosen to be edge-linear as well.
\end{lem}
\begin{proof}
Replacing a unimodal component $f_\md$ with $f^l_\md$ with its edge-linear interpolation, preserving the values at the vertices of $\mtree$, is again unimodal, and that the resulting components sum up to $f$, as their values at vertices do.
\end{proof}

As an immediate corollary we obtain the following.

\begin{lem}
\label{lem:vertices}
Any unimodal decomposition can be modified (without changing the number of components) so that all component modes are maximized at the vertices of $\mtree$.
\end{lem}

As a result of Lemmas \ref{lem:contraction}-\ref{lem:vertices}, we henceforth assume that all $f\in\Dist(\mtree)$ and all modes in unimodal decompositions thereof are edge-linear with modes at vertices.

\subsection{Free and forced}
\label{sec:mode-free}
Given the data $f\in\Dist(\mtree)$, we call a subtree $\mtree'\subset\mtree$ \style{mode-free} if there exists a unimodal decomposition of $f$ such that $\mtree'$ is free of modes. Clearly, mode-free subtrees are closed under the operation of taking subtrees. We will call a vertex of $\mtree$ \style{mode-forced} for $f$ if $v$ is a mode of every minimal unimodal decomposition of $f$ on $\mtree$.

\begin{lem}
\label{lem:forced}
For any $f\in\Dist(\mtree)$ satisfying the assumptions of \S\ref{sec:setup}, a mode-forced vertex exists.
\end{lem}
\begin{proof}
Choose any vertex $v$ and consider it the root of $\mtree$; then $\mtree$ becomes a union of several branches each having $v$ as a root. We will call a branch (a connected component of the complement of $v$, and all vertices in it excluding $v$) \style{insignificant}, if the values of $f$ are monotonically decreasing away from $v$ in that branch.

It is clear that removing all insignificant vertices does not increase the unimodal category of $f$ (but can generate more insignificant vertices). Pruning iteratively the tree of its insignificant vertices results either in a single vertex (in which case the original function $f$ is unimodal), or in a tree with at least two leaves, each of which is mode-forced.
\end{proof}

Specific examples of mode-forced vertices include the global maximum of $f$, as well as all local maxima of $\mtree$ such that all but one of the components of their complement are monotonically decreasing paths to leaves. 

\section{Sweeping}
\label{sec:sweep}
\subsection{Sweeping operation}
In this section we define a \style{sweeping operation} that will generate a function with a given mode.

Let $v$ be a vertex of $\mtree$ and $f\in\Dist(\mtree)$ satisfying the assumptions of \S\ref{sec:assumptions}. Define the function $h_{f,v}$ on $\mtree$ using the following procedure.

\begin{enumerate}
\item Orient all edges of $\mtree$ away from $v$ (making it the root of the oriented tree).
\item Set $h_{f,v}(v)=f(v)$.
\item If for an oriented edge, $u\longrightarrow w$, $h_{f,v}(u)$ is defined, then apply the \style{sweeping move}:
\begin{equation}
\label{eq:sweep}
h_{f,v}(w)
=
\left\{
\begin{array}{rl}
h_{f,v}(u) & \mbox{if} \quad f(u)<f(w);\\
\max(h_{f,v}(u)-(f(u)-f(w)),0) & \quad \mbox{otherwise} \\
\end{array}
\right .
\end{equation}
\end{enumerate}

Figure \ref{fig:propagate} below illustrates the sweeping mechanism.

\begin{figure}[htbp]
\begin{center}
\includegraphics[width=5in]{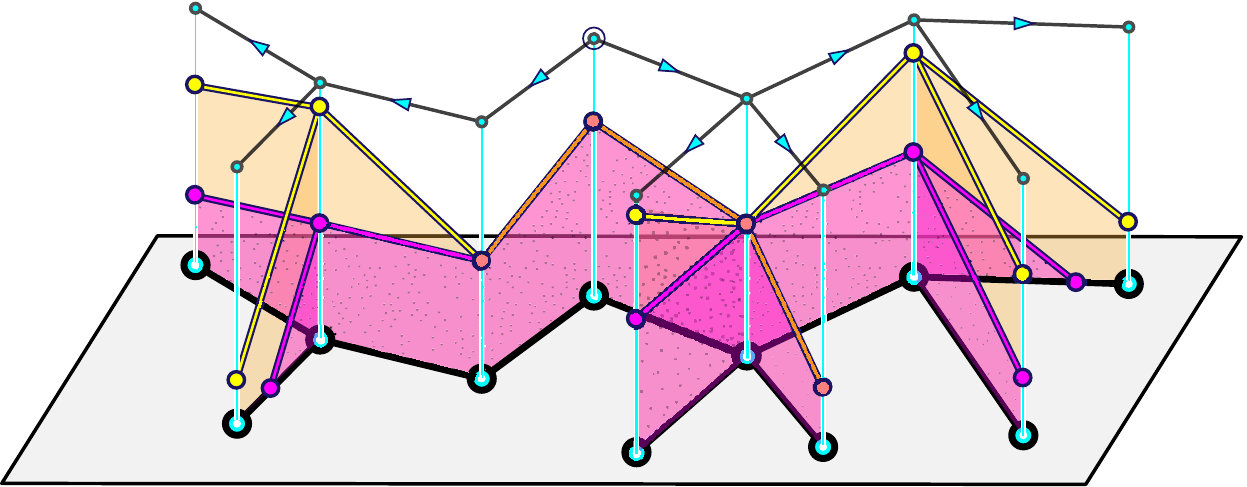}
\caption{Propagating $h_{f,v}$ (pink) given $f$ (yellow) from $v$ (the central mode). Note that the sweeping process can create new interior nodes in the trees where the function hits zero.}
\label{fig:propagate}
\end{center}
\end{figure}

\subsection{Remainders and freedom}
Given a vertex $v\in\mtree$, we will call the function $R_vf=f-h_{f,v}$ the $v$-\style{remainder}. This is nonnegative as, clearly, $h_{f,v}\leq f$ for any $v\in\mtree$. Also noted is that the support of $R_vf$ does not contain $v$.
The following is the critical result needed for our constructions.

\begin{thm}
\label{thm:nonshrinking}
If $\mtree'\subset\mtree$ is a mode-free subtree for $f$, and $v\notin \mtree'$, then $\mtree'$ is also mode-free for $R_vf$.
\end{thm}

In other words, taking $v$-remainders does not shrink free subtrees not containing $v$.
We postpone the proof until \S\ref{sec:stability}.
%

\section{Greedy Algorithm}
\label{sec:greedy}

The algorithm given in \cite{BG} for computing a minimal unimodal decomposition of $f\in\Dist([a,b])$ involved sweeping from $a$ to $b$, identifying mode-forced vertices, then removing their contribution by computing remainders, until the entire interval was swept.\footnote{That is, using the language of this paper.} In essence, the same process is here employed for a tree in the following \style{greedy algorithm}:
\begin{enumerate}
\item Using Lemma \ref{lem:forced}, find a mode-forced vertex $v$.
\item Construct the function $h_{f,v}$; it is a component of the unimodal decomposition.
\item Compute the remainder $R_vf$ and iterate.
\end{enumerate}

The detection of mode-forced vertices is constructive by Lemma \ref{lem:forced}.
Theorem \ref{thm:nonshrinking}, when proved, will permit iteration of the greedy algorithm by preserving the mode-free subtree structure: as the mode-free subtrees do not shrink, any unimodal decomposition will survive under sweeping, preserving the (strong) unimodal category.

\section{Remainders and leaf functions }
We need some preliminary results.
\subsection{Leaf functions}
Let $\mtree'$ be a mode-free subtree of $\mtree$. Denote by $\leaves$ the set of leaves of $\mtree'$ (not necessarily leaves of $\mtree$).

\begin{lem}
\label{lem:free_dec}
The condition of $\mtree'$ being mode-free for $f$ is equivalent to the existence of \style{leaf functions}, $f_\leaf\in\Dist(\mtree'), \leaf\in\leaves$, such that each $f_\leaf$ is non-increasing away from $\leaf$ and
\begin{equation}\label{eq:free_dec}
\sum_\leaf f_\leaf=f  \quad \rm{on}\quad  \mtree'.
\end{equation}
\end{lem}
\begin{proof}
If $\mtree'$ is mode-free for $f$, then $f=\sum_\alpha u_\alpha$ has a unimodal decomposition without modes on $\mtree'$. For each $\leaf\in\leaves$, let $f_\leaf$ be the sum of the $u_\alpha$ over all $\alpha$ whose modes lie in the connected component of $\mtree-\mtree'$ adjacent at leaf $\leaf$.
Likewise, given such a decomposition of $f$ on $\mtree'$ as a sum of leaf functions $f_\lambda$, choose a unimodal decomposition of $f$ on each connected component of $\mtree-\mtree'$. Use the restriction of this to each $\leaf\in\leaves$ to evenly divide the leaf functions into summands. The resulting decomposition is unimodal (perhaps not minimally so) and mode-free on $\mtree'$.
\end{proof}

\subsection{Remainders preserve mode-freeness}
\label{sec:stability}

Consider an edge $[u,w]$ in $\mtree'$. The restrictions of the  leaf functions to this edge can be modified so that the constraints of Lemma \ref{lem:free_dec} are still satisfied. In what follows, let $\leaf^*$ be the leaf adjacent to the component of $\mtree-\mtree'$ containing $v$.

We want to prove that $R_vf$ restricted to $\mtree'$ can still be decomposed as described in Lemma \ref{lem:free_dec}.
To achieve this we will modify the decomposition \eqref{eq:free_dec} edge by edge, away from $\leaf^*$ in such a way that $f_{\leaf^*}$ is always changed according to the sweeping rule \eqref{eq:sweep}, while the remaining functions $f_\leaf, \leaf\neq \leaf^*$ continue to satisfy the condition of being non-increasing away from $\leaf$.

It is immediate that if the restriction of $f_\leaf$ to the edge $[uw]$ is linear interpolating $f_\leaf(u)\geq f_\leaf(w)$, then any modification of $f_\leaf(u)$ to $g_\leaf(u)$ and $f_\leaf(w)$ to $g_\leaf(w)$ extends to an edge-linear, non-increasing  from $\leaf$ function $g_\leaf$ as long as $f_\leaf(u)\geq  g_\leaf(u)\geq  g_\leaf(w)\geq  f_\leaf(w)$, see Fig. \ref{fig:modify}. We will call any such modification of a function on an edge {\em admissible}.

\begin{figure}[htbp]
\begin{center}
\includegraphics[width=4.0in]{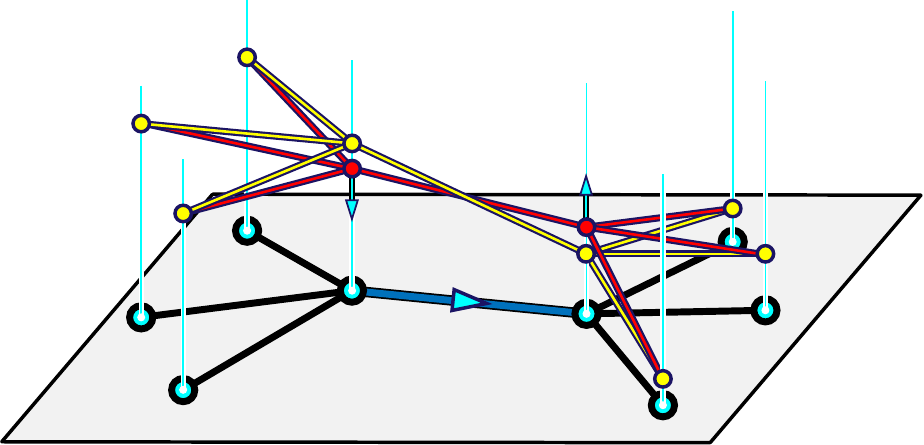}
\caption{Modifying a left-to-right non-increasing function along an edge locally preserves the semi-monotonicity.}
\label{fig:modify}
\end{center}
\end{figure}

The following lemma asserts that we can restrict our attention to admissible modifications of aggregations of leaf functions.

\begin{lem}
\label{lem:expand_adm}
Consider an edge, $[uw]$ and a collection of leaf functions $f_\leaf,\leaf\in\leaves$, each non-increasing from $u$ to $w$. Their sum $f_{\leaves}=\sum_{\leaf\in\leaves} f_\leaf$ is also non-increasing in the same direction. Let $g_{\leaves}$ be an admissible modification of $f_{\leaves}$. Then there exist admissible modifications $g_\leaf$ of each $f_\leaf, \leaf\in{\leaves}$ so that $\sum_{\leaf\in\leaves}g_\leaf=g_{\leaves}$ on $[u,w]$.
\end{lem}
\begin{proof}
By admissibility, we have that $f_{\leaves}(u)\geq g_{\leaves}(u)\geq g_{\leaves}(w)\geq f_{\leaves}(w)$. The claim is equivalent to existence of a solution to the following systems of (in)equalities on the variables $g_\leaf(\cdot)$:
\begin{equation}
\label{eq:primal}
\begin{array}{rcl}
f_\leaf(u) & \geq & g_\leaf(u),  \quad  \leaf\in\leaves \\
 g_\leaf(u) & \geq & g_\leaf(w),  \quad \leaf\in\leaves \\
 g_\leaf(w)& \geq &  f_\leaf(w), \quad \leaf\in\leaves \\
\sum_{\leaf\in\leaves}g_\leaf(u) &=& g_{\leaves}(u), \\
\sum_{\leaf\in\leaves}g_\leaf(w) &=& g_{\leaves}(w).
\end{array}
\end{equation}

Farkas' lemma implies that the solution to \eqref{eq:primal} {\em does not} exist if and only if the dual system (in variables $\mu^u_\leaf\geq 0, \mu_\leaf\geq 0, \mu^w_\leaf\geq 0, \nu_u, \nu_w$), corresponding to the lines in \eqref{eq:primal}) has a solution:
\begin{eqnarray}
-\mu^u_\leaf+\mu_\leaf-\nu_u &=&0;\\ \label{eq:dual1}
 \mu^w_\leaf-\mu_\leaf+\nu_w &=&0;\\ \label{eq:dual2}
\sum_\leaf \left((-\mu^u_\leaf f_\leaf(u)+\mu^w_\leaf f_\leaf(w)\right)-\nu_u g_{\leaves}(u)+\nu_w g_{\leaves}(w)&>&0. \label{eq:dual3}
\end{eqnarray}

Using the fact that $\sum_\leaf f_\leaf(u)=f_{\leaves}(u)$ and $\sum_\leaf f_\leaf(w)=f_{\leaves}(w)$, transform the left hand side of \eqref{eq:dual3} as
\begin{eqnarray}
\label{eq:dual3.1}
0<\sum_\leaf \left((-\mu^u_\leaf f_\leaf(u)+\mu^w_\leaf f_\leaf(w)\right)-\nu_u g_{\leaves}(u)+\nu_w g_{\leaves}(w)&=\\
\sum_\leaf \left((-\mu_\leaf+\nu_u) f_\leaf(u)+(\mu_\leaf-\nu_w) f_\leaf(w)\right)-\nu_u g_{\leaves}(u)+\nu_w g_{\leaves}(w)&=\\
\sum_\leaf \mu_\leaf(-f_\leaf(u)+f_\leaf(w))+\nu_u(f_{\leaves}(u)-g_{\leaves}(u))+\nu_w(g_{\leaves}(w)-f_{\leaves}(w)).
\label{eq:bound}
\end{eqnarray}

Further, \eqref{eq:dual1} implies that $\nu_u\leq \mu_\leaf$ for all $\leaf$, and therefore $\nu_u\leq \min \mu_\leaf=:\mu$. Similar reasoning (using \eqref{eq:dual2}) leads also to $\nu_w\leq\mu$.

Using these inequalities with \eqref{eq:bound}, and taking into account that $f_{\leaves}(u)\geq g_{\leaves}(u)\geq g_{\leaves}(w)\geq f_{\leaves}(w)$, we conclude that \eqref{eq:bound} is bounded above by
\begin{equation}
\mu(-f_{\leaves}(u)+f_{\leaves}(w))  +\mu(f_{\leaves}(u)-g_{\leaves}(u))+\mu(g_{\leaves}(w)-f_{\leaves}(w))\leq 0,
\end{equation}
contradicting the implication \eqref{eq:dual3.1} of non-existence of a solution to \eqref{eq:primal}.
\end{proof}

\section{Stability of mode-free subtrees}
In this section we finally prove Theorem \ref{thm:nonshrinking} on the stability of mode-free subtrees.

\begin{proof}[Proof of Theorem \ref{thm:nonshrinking}]
Lemmas \ref{lem:free_dec} and \ref{lem:expand_adm} imply that any admissible modification of $f$ on an edge of $\mtree'$ maintains a leaf function decomposition, and thus, the mode-free nature of $\mtree'$. To show that $R_vf$ preserves the mode-free subtree $\mtree'$ for $v\not\in\mtree'$, we apply the sweeping algorithm generating $h_{f,v}$ along the edges in $\mtree'$, propagating away from $\leaf^*$ (the leaf of $\mtree'$ adjacent to the component of $\mtree-\mtree'$ containing $v$) and performing admissible modifications.

Specifically, if we modify the leaf functions so that $f_{\leaf^*}=h_{f,v}+C$ for some constant $C\geq 0$, then we have a leaf function decomposition of $R_vf$ on $\mtree'$. If $f_{\leaf^*}(\leaf^*)<h_{f,v}(\leaf^*)$, then raising $f_{\leaf^*}(\leaf^*)$ to $h_{f,v}(\leaf^*)$ and lowering the other leaf functions is admissible. In this case, $C=0$. 

Fix an edge $[uw]$ where the leaf functions are to be modified. We assume that the propagation happens from $u$ to $w$ (that is $\leaf^*$ is in the component of $\mtree-(u,w)$ containing $u$). Consider the values $f_u=f(u)$ and $f_w=f(w)$ at the vertices. Decompose these as follows:
\begin{equation}
\label{eq:leafdecomps}
    f_u = f_u^* + f_u^\rightarrow + f_u^\leftarrow
    \quad
    f_w = f_w^* + f_w^\rightarrow + f_w^\leftarrow
\end{equation}
where $f^*_u=f_{\leaf^*}(u)$ and $f^*_w=f_{\leaf^*}(w)$ are values of the leaf function corresponding to $\leaf^*$; the $f^\rightarrow_u$, $f^\rightarrow_w$, $f^\leftarrow_u$, and $f^\leftarrow_w$ terms are values (at $u$ and $w$) of sums of the remaining leaf functions that are respectively oriented (nonincreasing) from $u$ to $w$ ($\rightarrow$) and from $w$ to $u$ ($\leftarrow$).
By Lemma \ref{lem:expand_adm}, it suffices to give admissible modifications of these aggregated leaf functions --- that is of values $f^\rightarrow_u$, $f^\rightarrow_w$, $f^\leftarrow_u$, and $f^\leftarrow_w$.

After potentially subdividing $\mtree$ based on where $h_{f,v}$ hits zero, there are two cases, corresponding to $f$ being increasing or decreasing from $u$ to $v$. If $f$ is decreasing, then $f_u-f_w\leq f_u^*-f_w^*$, so that its decrement is less or equal than the decrement of $f_{\leaf^*}$.
For these cases, we modify $f$-values of leaf functions at $u$ and $w$ to $g$-values as per Lemma \ref{lem:expand_adm}. These modifications are given in Table \ref{tab:mod} and can be checked as admissible and summing to $f$ as per (\ref{eq:leafdecomps}).

\begin{center}
\begin{tabular}{|c||c|c|c||c|c|c|}
  \hline
  CASE & $g_u^\rightarrow$ & $g_u^\leftarrow$ & $g_u^*$ & $g_w^\rightarrow$ & $g_w^\leftarrow$ & $g_w^*$  \\ \hline\hline
  $f_u-f_v\geq 0$   & $f_u^\rightarrow$  & $f_u^\leftarrow$ & $f_u^*$ & $f_u^\rightarrow$ & $f_w-f_u+f_u^\leftarrow$ & $f_u^*$  \\ \hline
  $f_u-f_v<0$   & $f_u^\rightarrow$  & $f_u^\leftarrow$ & $f_u^*$ & $f_u^\rightarrow$ & $f_u^\leftarrow$ & $f_u^*-f_u+f_w$  \\ \hline
\end{tabular}
\label{tab:mod}
\end{center}
Sweeping over $\mtree'$, this modifies the leaf function at $\leaf^*$ to be of the form $h_{f,v}+C$. Thus, modifying $f$ to the remainder $R_vf$ has the effect of maintaining a leaf function decomposition on $\mtree'$.
\end{proof}

\begin{cor}
Under the assumptions on $f\in\Dist(\mtree)$ from \S\ref{sec:assumptions}, the greedy algorithm of \S\ref{sec:greedy} returns a minimal unimodal decomposition.
\end{cor}

\section{Conclusions}
The present work gives a constructive method for the computation of minimal unimodal decompositions of densities on trees. We end with a few remarks.
\begin{enumerate}
\item
Unimodal decompositions are far from unique. Minimal unimodal decompositions are likewise not unique, but in a structured manner. The mode-forced vertices and mode-free subtrees provide a skeleton on which a given density hangs. By performing the sweeping moves of the greedy algorithm in different ``directions'' one arrives at many minimal unimodal decompositions. The analogous variations over an interval consist in sweeping from left-to-right or right-to-left \cite{BG}.
\item
Assume that the values of $f$ on the $V$ vertices of $\mtree$ are given as the input. As written, the greedy algorithm for generating a unimodal decomposition is $O(V\ucat(\dist))$. One can do better. For example, if the modes have supports of uniformly bounded size $S$, the algorithm runtime drops to $O(S \ucat(f))$. 
\item
Of course, the restriction of these results to trees is suboptimal. For applications of unimodal decompositions in disciplines where precise geometric data can be hard to come by (e.g., phylogenetics or neuroscience), graphs with cycles are not only possible but (especially in the case of neuroscience) critical. The generalization from graphs to higher dimensional domains is likewise important but appears formidable, depending on the model of unimodal decomposition employed. We view the construction of a minimal unimodal decompositions on graphs to be an important prerequisite for this challenge.
\end{enumerate}

\paragraph{\bf Acknowledgements.} This research was done while YB was visiting the Departments of Mathematics and ESE of the University of Pennsylvania - the hospitality of both departments is warmly appreciated.


\begin{thebibliography}{1}

\bibitem{BG} Y. Baryshnikov and R. Ghrist, ``Unimodal Category
and Topological Statistics,'' {\em Proc. NOLTA}, 2011.

\bibitem{B}
J. Behboodian, ``On the modes of a mixture of two normal
distributions,''{\em Technometrics}, 12, 1970, 131--139.

\bibitem{CM}
G. Carlsson \& F. M{\'e}moli, ``Classifying clustering schemes,''
  \emph{Found. Comput. Math.}, 13~(2):221--252 (2013).

\bibitem{CW}
M. Carreira-Perpinan and C. Williams, ``On the number of modes of a
Gaussian mixture,'' in {\em Scale-Space Methods in Computer Vision,
Lecture Notes in Comput. Sci.}, vol. 2695, 2003, 625--640.

\bibitem{Cornea+}
O. Cornea, G. Lupton, J. Oprea, and D. Tanr\'e, {\em
Lusternik-Schnirelmann Category}, Amer. Math. Soc., 2003.

\bibitem{E} I. Eisenberger, ``Genesis of bimodal
distributions,'' {\em Technometrics}, 6, 1964, 357--363.

\bibitem{Farber}
M. Farber, ``Topological complexity of motion planning,'' {\em Discrete
  Comput. Geom.}, 29~(2), 2003, 211--221.

\bibitem{EAT}
R. Ghrist, {\em Elementary Applied Topology}, Createspace, 2014.

\bibitem{Govc}
D. Govc, ``Unimodal category and the monotonicity conjecture,'' arXiv:1709.06547 [math.AT], 2017.

\bibitem{Huntsman}
S. Huntsman, ``Topological Mixture Estimation,'' preprint, 2018.

\bibitem{Ka}
I. Kakiuchi, ``Unimodality conditions of the distribution of a
mixture of two distributions,'' {\em Math. Sem. Notes Kobe Univ.},
9, 1981, 315--325.

\bibitem{Ke}
J. Kemperman, ``Mixtures with a limited number of modal intervals,''
{\em Ann. Statist.} 19, 1991, 2120--2144.

\bibitem{Kleinberg}
J. Kleinberg , ``An impossibility theorem for clustering,'' in {\em Proc.
  NIPS}, 2002, 446--453.

\bibitem{RF}
C. Robertson and J. Fryer, ``Some descriptive properties of normal
mixtures,'' {\em Skand. Aktuarietidskr}, 1969, 137--146.

\end{thebibliography}
\end{document}